\documentclass[11pt]{amsart}
\usepackage{amsmath,amsthm,amsfonts,amssymb,latexsym,epsfig}

\newtheorem{theorem}{Theorem}[section]

\newtheorem{lemma}{Lemma}[section]

\newtheorem{cor}{Corollary}[section]

\numberwithin{equation}{section}

\theoremstyle{definition}

\theoremstyle{remark}

\begin{document}
\title{Finite Sections of Weighted Hardy's Inequality}
\author{Peng Gao}
\address{Department of Computer and Mathematical Sciences,
University of Toronto at Scarborough, 1265 Military Trail, Toronto
Ontario, Canada M1C 1A4} \email{penggao@utsc.utoronto.ca}
\date{December 11, 2007.}


\begin{abstract}
 We study finite sections of weighted Hardy's inequality following the approach of De Bruijn. Similar to the unweighted case, we obtain an asymptotic expression for the optimal constant.
\end{abstract}

\maketitle
\section{Introduction}
\label{sec 1} \setcounter{equation}{0}
   Suppose throughout that $p\neq 0, \frac{1}{p}+\frac{1}{q}=1$.
   Let $l^p$ be the Banach space of all complex sequences ${\bf a}=(a_n)_{n \geq 1}$ with norm
\begin{equation*}
   ||{\bf a}||: =\Big(\sum_{n=1}^{\infty}|a_n|^p\Big)^{1/p} < \infty.
\end{equation*}
  The celebrated
   Hardy's inequality (\cite[Theorem 326]{HLP}) asserts that for $p>1$,
\begin{equation}
\label{eq:1} \sum^{\infty}_{n=1}\Big{|}\frac {1}{n}
\sum^n_{k=1}a_k\Big{|}^p \leq \Big(\frac
{p}{p-1} \Big )^p\sum^\infty_{k=1}|a_k|^p.
\end{equation}

   Hardy's inequality can be regarded as a special case of the
   following inequality:
\begin{equation*}
\label{01}
   \sum^{\infty}_{j=1}\Big{|}\sum^{\infty}_{k=1}c_{j,k}a_k
   \Big{|}^p \leq U \sum^{\infty}_{k=1}|a_k|^p,
\end{equation*}
   in which $C=(c_{j,k})$ and the parameter $p$ are assumed
   fixed ($p>1$), and the estimate is to hold for all complex
   sequences ${\bf a}$. The $l^{p}$ operator norm of $C$ is
   then defined as the $p$-th root of the smallest value of the
   constant $U$:
\begin{equation*}
\label{02}
    ||C||_{p,p}=U^{\frac {1}{p}}.
\end{equation*}

    Hardy's inequality thus asserts that the Ces\'aro matrix
    operator $C$, given by $c_{j,k}=1/j , k\leq j$ and $0$
    otherwise, is bounded on {\it $l^p$} and has norm $\leq
    p/(p-1)$. (The norm is in fact $p/(p-1)$.)

    We say a matrix $A$ is a summability matrix if its entries satisfy:
    $a_{j,k} \geq 0$, $a_{j,k}=0$ for $k>j$ and
    $\sum^j_{k=1}a_{j,k}=1$. We say a summability matrix $A$ is a weighted
    mean matrix if its entries satisfy:
\begin{equation*}
    a_{j,k}=\lambda_k/\Lambda_j,  ~~ 1 \leq k \leq
    j; \Lambda_j=\sum^j_{i=1}\lambda_i, \lambda_i \geq 0, \lambda_1>0.
\end{equation*}
  
    Hardy's inequality \eqref{eq:1} motivates one to
    determine the $l^{p}$ operator norm of an arbitrary summability matrix $A$.
We refer the readers to the articles \cite{G3}, \cite{G5} and the references therein for recent progress in this direction. 

   From now on we will assume $a_n \geq 0$ for $n \geq 1$ and any
   infinite sum converges. We note here by a change of variables $a_k \rightarrow a^{1/p}_k$ in \eqref{eq:1}, we obtain the following well-known Carleman's inequality as the limiting case of Hardy's inequality on letting $p \rightarrow +\infty$:
\begin{equation*}
   \sum^\infty_{n=1}\Big(\prod^n_{k=1}a_k \Big)^{\frac 1{n}}
\leq e\sum^\infty_{n=1}a_n,
\end{equation*}
   with the constant $e$ best possible.

  There is a rich literature on many different proofs of Hardy's and Carleman's inequality as well as their generalizations and extensions. We shall refer the readers to the survey articles  \cite{P&S}, \cite{D&M} and \cite{KMP} as well as the references therein for an account of Hardy's and Carleman's inequality. 

   In \cite{De}, De Bruijn studied finite sections of Carleman's inequality:
\begin{equation*}
   \sum^N_{n=1}\Big(\prod^n_{k=1}a_k \Big)^{\frac 1{n}}
\leq \mu_N\sum^N_{n=1}a_n.
\end{equation*}   
   where $N \geq 1$ is any integer. He showed that the best constant satisfies
\begin{equation*}
   \mu_N=e-\frac {2\pi^2e}{(\log N)^2}+ O \Big (\frac 1{(\log N)^3} \Big ).
\end{equation*}  
 
  De Bruijn's result was generalized by Ackermans for the case of finite sections of Hardy's inequality:
\begin{equation*}
   \sum^{N}_{n=1}\Big{|}\frac {1}{n}
\sum^n_{k=1}a_k\Big{|}^p \leq \mu_N\sum^N_{k=1}|a_k|^p,
\end{equation*}
where $N \geq 1$ is any integer and by an abuse of notation, we use the same symbol $\mu_N$ here to denote the best constant that makes the above inequality hold. 
  Ackermans \cite{A} showed that the best constant satisfies ($p>1$)
\begin{equation*}
   \mu^{1/p}_N=q-\frac {2\pi^2q^2/p}{(\log N)^2}+ O \Big (\frac 1{(\log N)^3} \Big ).
\end{equation*}  
 
   We point out here that in the case of $p=2$, one can also treat finite sections of Hardy's inequality on relating it to eigenvalues of certain symmetric matrices. More generally, using results of Widom \cite{Wi1}, \cite{Wi2}, Wilf \cite{W1, W2} obtained similar results when such matrices are generalized by a function $K(x,y)$, with $K(x,y)$ being symmetric, non-negative for non-negative $x$ and $y$, homogeneous of degree $-1$, and decreasing. We note here special cases of Wilf's result include the Hilbert matrix, given by $K(x,y)=1/(x+y)$, which was first studied by De Bruijn and Wilf in \cite{D&W}. Another case is the matrix corresponding to Hardy's inequality, as one can show that it is similar to the matrix given by $K(x,y)=1/\max(x,y)$ (see \cite{W1}). We remark here this approach only gives a weaker result compared to the result of Ackermans above. Bolmarcich \cite{Bol} further extended Wilf's result to the case $p>1$ but his result is less precise.

   Motivated by the above results of De Bruijn and Ackerman, it is our goal in this paper to study finite sections of weighted Hardy's inequality:
\begin{equation}
\label{1}
   \sum^{N}_{n=1}\Big{|}\sum^{n}_{k=1} \frac {\lambda_ka_k}{\Lambda_n}
   \Big{|}^p \leq \mu_N \sum^{N}_{n=1}|a_n|^p.
\end{equation}

    We note here in \cite{G6}, the author has obtained asymptotic expressions for $\mu_N$ for the weighted Carleman's inequality under certain conditions, following De Bruijn's approach in \cite{De}. This maybe regarded as the limiting case $p \rightarrow +\infty$ of weighed Hardy's inequality, following our discussions above. The structure of the paper is similar to that of \cite{G6} and we point out here that what we have in mind are the cases when $\lambda_k=k^{\alpha}$ and for this reason we shall prove the following
\begin{theorem}
\label{thm2}
 Let $p \geq 2$ be fixed and $\{ \lambda_k \}^{\infty}_{k=1}$ be a non-decreasing positive sequence satisfying 
\begin{eqnarray}
\label{5}
 && L = \sup_n\Big(\frac {\Lambda_{n+1}}{\lambda_{n+1}}-\frac
    {\Lambda_n}{\lambda_n}\Big) \leq 1, \\ 
\label{1.3}
 && \sup_k \frac {\lambda_{k+1}}{\lambda_k} < +\infty, \\
\label{1.3'}
  & &  \frac {\lambda_{k}}{\lambda_{k+1}} = 1-\frac {1/L-1}{k}+O(\frac 1{k^2}), \\
\label{3.5'}
 & & \frac {\lambda_k}{\Lambda_k} = \frac {1}{Lk}+O(\frac 1{k^2}), \\
\label{1.6}
 & & \inf_k\Big (\frac {\Lambda_{k+1}}{\lambda_{k+1}}-\frac {\Lambda_k}{\lambda_k} \Big ) > 0, \\
\label{1.4}
   & & 1-L/p \geq  \frac {\Lambda_{k}}{\lambda_{k}}\Big (1-\Big (\frac {\Lambda_k}{\Lambda_{k+1}}\Big )^{\frac {p-1}{p}}\Big (\frac {\lambda_k}{\lambda_{k+1}}\Big )^{\frac {1}{p}}  \Big ).
\end{eqnarray}
  Then as $N \rightarrow +\infty$, inequality \eqref{1} holds with the best constant satisfying:
\begin{equation*}
   \mu_N=(1-L/p)^{-p} - \frac {2L^2\pi^2(1-L/p)^{-2-p}}{q(\log N)^2}+O\Big (\frac 1{(\log N)^3} \Big ).
\end{equation*}  
\end{theorem}

   We note here that the case $k=1$ of \eqref{1.6} implies $L>0$, which we shall use without further mentioning throughout the paper. Note also that this makes \eqref{1.3'} and  \eqref{3.5'} meaningful. We may also assume $N \geq 2$ from now on.
\section{Preliminary Treatment}
\label{sec 2} \setcounter{equation}{0}
   In this section, we summarize some of the proof in \cite{G5} that gives an upper bound for the number $\mu_N$ appearing in \eqref{1} assuming \eqref{5}, a result first obtained by Cartlidge \cite{Car}. We refer the reader to \cite{G5} for more details on our discussions below. Let
\begin{equation*}
   A_n=\sum^n_{k=1}\frac {\lambda_ka_k}{\Lambda_n},
\end{equation*}
   our goal is to determine the maximum value $\mu_N$ of $\sum^N_{n=1}A^p_n$ subject to the constraint $\sum^N_{n=1}a^p_n=1$ here. It is shown in \cite{G5} that it suffices to consider the case $a_n > 0$ for all $1 \leq n \leq N$ when the maximum is reached. We now define
\begin{equation*}
  F({\bf a}; \mu)=\sum^N_{n=1}A^p_n-\mu (\sum^N_{n=1}a^p_n-1),
\end{equation*}
  where ${\bf a}=(a_n)_{1 \leq n \leq N}$. By the Lagrange method, we need to solve $\nabla F=0$, which turns out to yield a recurrence relation starting with
$\Omega_1(\mu)=1/\mu$ and
\begin{equation*}
   \Omega^{\frac 1{p-1}}_{k+1}(\mu)=\frac {\Lambda_{k}}{\Lambda_{k+1}}\Big (\frac {\Omega_{k}(\mu)}{\frac {\lambda_{k+1}}{\Lambda_k}(\Lambda_k/\lambda_k-\Omega_{k}(\mu))} \Big )^{\frac 1{p-1}}+\frac {\lambda_{k+1}}{\Lambda_{k+1}} \Big ( \frac 1{ \mu} \Big )^{\frac 1{p-1}}.
\end{equation*}
   It is then shown in \cite{G5} that if \eqref{5} is satisfied and $\mu \geq (1-L/p)^{-p}$, then
for $1 \leq n \leq N$,
\begin{equation}
\label{2.3}
  \frac {\lambda_{n}}{\Lambda_{n}} \Big ( \frac 1{ \mu} \Big )^{\frac 1{p-1}}\leq \Omega^{\frac 1{p-1}}_n(\mu) \leq \Big(\frac {p-L}{p} \Big )^{1/(p-1)}-\frac {L}{p}\Big(\frac {p-L}{p}\Big)^{1/(p-1)}\Big (\frac {\lambda_n}{\Lambda_n} \Big ).
\end{equation}
   On the other hand, it is also shown that $\Omega_N(\mu_N)=\Lambda_{N}/\lambda_{N}$, and this forces $\mu_N < (1-L/p)^{-p}$.
To facilitate our approach in what follows, we now write $\nu=1/\mu$ and define a new sequence $h_k(\nu)$ by
\begin{equation*}
  h_k(\nu)=\Omega^{\frac 1{p-1}}_{k}(1/\nu)- \frac {\lambda_k}{\Lambda_k}\nu^{\frac 1{p-1}}.
\end{equation*}
  It follows that $h_1(\nu)=0$ and 
\begin{equation}
\label{2.2}
   h_{k+1}(\nu)=\frac {\Lambda_{k}}{\Lambda_{k+1}}\Big (\frac {\lambda_k}{\lambda_{k+1}}\Big)^{\frac 1{p-1}}\Big(h_{k}(\nu)+\frac {\lambda_k}{\Lambda_k}\nu^{\frac 1{p-1}}\Big)\Big (1-\frac {\lambda_{k}}{\Lambda_k}\Big(h_{k}(\nu)+\frac {\lambda_k}{\Lambda_k}\nu^{\frac 1{p-1}}\Big)^{p-1} \Big )^{-\frac 1{p-1}}.
\end{equation}
   We note that it follows from \eqref{2.3} that for 
$1 \leq n \leq N$ (with $\Lambda_0=0$ here)
\begin{equation}
\label{2.4}
    0 \leq h_{n}\Big(  \Big(\frac {p-L}{p} \Big )^{p}  \Big ) \leq \frac {\Lambda_{n-1}}{\Lambda_n} \Big(\frac {p-L}{p} \Big )^{1/(p-1)}.
\end{equation}

\section{The Breakdown Index}
\label{sec 3} \setcounter{equation}{0}
    As in \cite{De}, we now try to evaluate $h_k(1/\mu)$ consecutively from \eqref{2.2} for any $\mu>0$, starting with $h_1=0$. Certainly we are only interested in the real values of $h_k$ and hence we say that the procedure breaks down at the first $k$ where 
\begin{equation*}
  1-\frac {\lambda_{k}}{\Lambda_k}\Big(h_{k} \Big(\frac {1}{\mu} \Big)+\frac {\lambda_{k}}{\Lambda_{k}} \Big ( \frac 1{ \mu} \Big )^{\frac 1{p-1}}\Big)^{p-1} \leq 0.
\end{equation*}
 Or equivalently, 
\begin{equation}
\label{3.1}
 h_k \Big(\frac {1}{\mu} \Big) \geq \Big(\frac {\Lambda_k}{\lambda_{k}}\Big )^{\frac 1{p-1}}-\frac {\lambda_{k}}{\Lambda_{k}} \Big ( \frac 1{ \mu} \Big )^{\frac 1{p-1}}.
\end{equation} 
   We define the breakdown index $N_{\mu}$ as the smallest $k$ for which inequality \eqref{3.1} holds if there is such a $k$ and we put $N_{\mu} = +\infty$ otherwise. Thus for all $\mu>0$ we can say that $h_k(1/\mu)$ is defined for all $k \leq N_{\mu}$.

    Note that \eqref{2.3} implies $N_{\mu} = +\infty$ when $\mu \geq (1-L/p)^{-p}$.  So from now on we may assume $0 < \mu < (1-L/p)^{-p}$ and it is convenient to have some monotonicity properties available in this case. We have $h_1(1/\mu)=0$ for $0 < \mu < (1-L/p)^{-p}$ and we let $\mu_1$ be the largest $\mu$ for which inequality \eqref{3.1} holds for $k=1$, this implies $\mu_1=1$. Now $h_2(1/\mu)$ is defined for $\mu > \mu_1$, and $h_2(1/\mu)$ is given by \eqref{2.2} as
\begin{equation*}
   h_{2}(1/\mu)=\frac {\Lambda_{1}}{\Lambda_{2}}\Big (\frac {\lambda_1}{\lambda_{2}}\Big)^{\frac 1{p-1}}(\mu- 1 )^{-\frac 1{p-1}},
\end{equation*}
 which is a decreasing function of $\mu$ for $\mu > \mu_1$. Note also that the right-hand side expression of inequality \eqref{3.1} is an increasing function of $\mu$ for any fixed $k$. It follows from \eqref{2.4} that
\begin{equation*}
  \lim_{\mu \rightarrow \mu_1^+}h_2(1/\mu)=+\infty; \hspace{0.1in} h_2\Big(\Big(\frac {p-L}{p} \Big )^{p}\Big ) \leq \frac {\Lambda_{1}}{\Lambda_2} \Big(\frac {p-L}{p} \Big )^{1/(p-1)} < \Big(\frac {\Lambda_2}{\lambda_{2}}\Big )^{\frac 1{p-1}}-\frac {\lambda_{2}}{\Lambda_{2}} \Big(\frac {p-L}{p} \Big )^{\frac p{p-1}}.
\end{equation*}
   Thus there is exactly one value of $\mu_1< \mu < (1-L/p)^{-p}$ for which inequality \eqref{3.1} holds with equality for $k=2$ and we define this value of $\mu$ to be $\mu_2$. This procedure can be continued. At each step we argue that $h_k(1/\mu)$ is defined and decreasing for $\mu > \mu_{k-1}$, that 
\begin{equation*}
  \lim_{\mu \rightarrow \mu_{k-1}^+}h_k(1/\mu)=+\infty; \hspace{0.1in} h_k\Big(\Big(\frac {p-L}{p} \Big )^{p}\Big ) \leq \frac {\Lambda_{k-1}}{\Lambda_k} \Big(\frac {p-L}{p} \Big )^{1/(p-1)} < \Big(\frac {\Lambda_k}{\lambda_{k}}\Big )^{\frac 1{p-1}}-\frac {\lambda_{k}}{\Lambda_{k}} \Big(\frac {p-L}{p} \Big )^{\frac 1{p-1}}.
\end{equation*}
   We then infer that $\mu_k$ is uniquely determined by 
\begin{equation*}
 h_k \Big(\frac {1}{\mu_k} \Big) = \Big(\frac {\Lambda_k}{\lambda_{k}}\Big )^{\frac 1{p-1}}-\frac {\lambda_{k}}{\Lambda_{k}} \Big ( \frac 1{ \mu_k} \Big )^{\frac 1{p-1}}.
\end{equation*} 
  Moreover, $h_{k+1}(1/\mu)$ is again defined and decreasing for $\mu > \mu_{k}$ as both terms on the right of \eqref{2.2} involving $\nu=1/\mu$ are non-negative decreasing functions of $\mu$.
  
   Thus by induction we obtain that
\begin{equation}
\label{3.1'}
   1=\mu_1 < \mu_2 < \mu_3 < \ldots < (1-L/p)^{-p},
\end{equation}
   and that $h_{k+1}(1/\mu)$ is defined and decreasing for $\mu > \mu_{k}$. Moreover, 
\begin{equation*}
 h_k \Big(\frac {1}{\mu} \Big) > \Big(\frac {\Lambda_k}{\lambda_{k}}\Big )^{\frac 1{p-1}}-\frac {\lambda_{k}}{\Lambda_{k}} \Big ( \frac 1{ \mu} \Big )^{\frac 1{p-1}}
\end{equation*} 
  if $\mu_{k-1} < \mu < \mu_k$, 
\begin{equation*}
  h_k \Big(\frac {1}{\mu_k} \Big) = \Big(\frac {\Lambda_k}{\lambda_{k}}\Big )^{\frac 1{p-1}}-\frac {\lambda_{k}}{\Lambda_{k}} \Big ( \frac 1{ \mu_k} \Big )^{\frac 1{p-1}}
\end{equation*} 
   and
\begin{equation*}
 h_k \Big(\frac {1}{\mu} \Big) < \Big(\frac {\Lambda_k}{\lambda_{k}}\Big )^{\frac 1{p-1}}-\frac {\lambda_{k}}{\Lambda_{k}} \Big ( \frac 1{ \mu_k} \Big )^{\frac 1{p-1}}
\end{equation*} 
  if $ \mu > \mu_k$.

  It follows that the breakdown index $N_{\mu}$ equals $1$ if $\mu \leq \mu_1$, $2$ if $\mu_1 < \mu \leq \mu_2$, etc. We remark here that for fixed $\mu \leq (1-L/p)^{-p}$, the $h_k(1/\mu)$'s are non-negative and increase as $k$ increases from $1$ to $N_k$. This follows from \eqref{2.2} by noting that
\begin{eqnarray}
\label{3.2}
 &&  h_{k+1}(\nu)-h_k(\nu) \\
&=& \frac {\Lambda_{k}}{\Lambda_{k+1}}\Big (\frac {\lambda_k}{\lambda_{k+1}}\Big)^{\frac 1{p-1}}\Big(h_{k}(\nu)+\frac {\lambda_k}{\Lambda_k}\nu^{\frac 1{p-1}}\Big)\Big (1-\frac {\lambda_{k}}{\Lambda_k}\Big(h_{k}(\nu)+\frac {\lambda_k}{\Lambda_k}\nu^{\frac 1{p-1}}\Big)^{p-1} \Big )^{-\frac 1{p-1}}-h_k(\nu). \nonumber
\end{eqnarray}
   It thus suffices to show the right-hand side expression above is non-negative. Equivalently, this is $0 \leq f_k(h_k(\nu)+\lambda_k\nu^{\frac 1{p-1}}/\Lambda_k)$, where
\begin{equation*}
 f_k(x)= \frac {\Lambda_{k}}{\Lambda_{k+1}}\Big (\frac {\lambda_k}{\lambda_{k+1}}\Big)^{\frac 1{p-1}}x\Big (1-\frac {\lambda_{k}}{\Lambda_k}x^{p-1} \Big )^{-\frac 1{p-1}}-x+\frac {\lambda_k\nu^{\frac 1{p-1}}}{\Lambda_k}.
\end{equation*}
  It is easy to see that $f_k(x)$ is minimized at $x=x_0$ which satisfies
\begin{equation*}
 \frac {\Lambda_{k}}{\Lambda_{k+1}}\Big (\frac {\lambda_k}{\lambda_{k+1}}\Big)^{\frac 1{p-1}}\Big (1-\frac {\lambda_{k}}{\Lambda_k}x^{p-1}_0 \Big )^{-\frac p{p-1}}=1.
\end{equation*}
   It follows that
\begin{equation*}
 f_k(x) \geq f_k(x_0)=\frac {\lambda_{k}}{\Lambda_{k}}(\nu^{\frac 1{p-1}}-x^{p}_0).
\end{equation*}
   Thus it suffices to check $f_k(x_0) \geq 0$ or equivalently,
\begin{equation*}
   \nu^{1/p} \geq x^{p-1}_0=\frac {\Lambda_{k}}{\lambda_{k}}\Big (1-\Big (\frac {\Lambda_k}{\Lambda_{k+1}}\Big )^{\frac {p-1}{p}}\Big (\frac {\lambda_k}{\lambda_{k+1}}\Big )^{\frac {1}{p}}  \Big ).
\end{equation*}
   Note that $\nu \geq (1-L/p)^{p}$  and it follows from \eqref{1.4} that the above inequality holds, which implies that $f_k(x_0) \geq 0$ so that the $h_k(1/\mu)$'s increase as $k$ increases from $1$ to $N_k$. 

  The breakdown condition \eqref{3.1} is slightly awkward. We now replace it by a simpler one in the case of $p \geq 2$, for example, $h_k > C_0$ (to be determined in what follows), by virtue of the following argument. Let $0 < \mu < (1-L/p)^{-p}$ and assume that $N_0$ is the smallest integer such that $h_{N_0} > C_0$. Note that \eqref{3.5'} implies that $\lim_{k \rightarrow +\infty}\Lambda_k/\lambda_k = +\infty$ so that the right-hand side expression of \eqref{3.1} approaches $+\infty$ as $k$ tends to $+\infty$. Hence we may assume $N_{\mu} \geq N_0$ without loss of generality. Then we have
\begin{equation*}
  \log N_{\mu}- \log N_0 =O(1).
\end{equation*}
   For, if $N_0 \leq k \leq N_{\mu}$, the right-hand side of \eqref{3.2} equals (with $h_k=h_k(\nu)$ here)
\begin{eqnarray}
  && \frac {\Lambda_{k}}{\Lambda_{k+1}}\Big (\frac {\lambda_k}{\lambda_{k+1}}\Big)^{\frac 1{p-1}}\Big(h_{k}+\frac {\lambda_k}{\Lambda_k}\nu^{\frac 1{p-1}}\Big)\Big (1-\frac {\lambda_{k}}{\Lambda_k}\Big(h_{k}+\frac {\lambda_k}{\Lambda_k}\nu^{\frac 1{p-1}}\Big)^{p-1} \Big )^{-\frac 1{p-1}}-h_k  \nonumber \\
\label{4.0}
& \geq &  \frac {\Lambda_{k}}{\Lambda_{k+1}}\Big (\frac {\lambda_k}{\lambda_{k+1}}\Big)^{\frac 1{p-1}}\Big(h_{k}+\frac {\lambda_k}{\Lambda_k}\nu^{\frac 1{p-1}}\Big)\Big (1+\frac {\lambda_{k}}{(p-1)\Lambda_k}\Big(h_{k}+\frac {\lambda_k}{\Lambda_k}\nu^{\frac 1{p-1}}\Big)^{p-1} \Big )-h_k \\
& \geq &  \frac {\lambda_{k}}{(p-1)\Lambda_{k+1}}\Big (\frac {\lambda_k}{\lambda_{k+1}}\Big)^{\frac 1{p-1}}\Big(h_{k}+\frac {\lambda_k}{\Lambda_k}\nu^{\frac 1{p-1}}\Big)^p 
+\Big (\frac {\Lambda_{k}}{\Lambda_{k+1}}\Big (\frac {\lambda_k}{\lambda_{k+1}}\Big)^{\frac 1{p-1}}-1 \Big ) \Big(h_{k}+\frac {\lambda_k}{\Lambda_k}\nu^{\frac 1{p-1}}\Big). \nonumber
\end{eqnarray}
  It follows from this and \eqref{1.3'} and \eqref{3.5'} that there exists a constant $C_0>1$, independent of $k$ but may depend on $p$ and an integer $N_1$ independent of $\mu$  such that when $h_{k}(\nu) \geq C_0$, for $k \geq N_1$, we have
\begin{equation*}
  h_{k+1}(1/\mu)-h_k(1/\mu) \geq \frac {C_1\lambda_{k}}{\Lambda_{k+1}}h^2_k(1/\mu),
\end{equation*}
  for some positive constant $C_1>0$.
  We may assume $N_0 \geq N_1$ from now on without loss of generality and we now simplify the above relations by defining $d_{N_0}, d_{N_0+1}, \ldots$, starting with $d_{N_0}=h_{N_0}$, and
\begin{equation}
\label{3.3}
  d_{k+1}-d_k= \frac {C_1\lambda_{k}}{\Lambda_{k+1}}d^2_k.
\end{equation}
  
   Obviously we have 
\begin{equation*}
 d_k \leq h_k \leq  \Big(\frac {\Lambda_k}{\lambda_{k}}\Big )^{\frac 1{p-1}}-\frac {\lambda_{k}}{\Lambda_{k}} \Big ( \frac 1{ \mu} \Big )^{\frac 1{p-1}} \leq \Lambda_k/\lambda_k.
\end{equation*} 
for $N_0 \leq k \leq N_{\mu}$. It follows from \eqref{3.3} that 
\begin{equation*}
  d_{k+1}-d_k \leq C_1d_k.
\end{equation*}
   The above implies that we have $d_{k+1} \leq (C_1+1)d_k$ for $N_0 \leq k \leq N_{\mu}$ and \eqref{3.3} further implies that
\begin{equation}
\label{3.5}
   d_{k+1}-d_k \geq \frac {C_1\lambda_{k}}{(C_1+1)\Lambda_{k+1}}d_kd_{k+1}.
\end{equation}
   We now apply \eqref{3.5'} to obtain via \eqref{3.5} that there exists a constant $C_2>0$ and an integer $N_2$ independent of $\mu$ such that for $k \geq N_2$, 
\begin{equation*}
   d^{-1}_{k}-d^{-1}_{k+1} \geq \frac {C_2}{k+1}.
\end{equation*}
   Certainly we may assume $N_0 \geq N_2$ as well. Summing the above for $N_0 \leq k \leq N_{\mu}-1$ yields:
\begin{equation*}
   \frac {1}{C_0} \geq d^{-1}_{N_0} \geq \sum_{N_0 \leq k \leq N_{\mu}-1}\frac {C_2}{k+1}.
\end{equation*}
   It follows from this that
\begin{equation}
\label{3.6}
  \log N_{\mu}- \log N_0=-\sum_{N_0 \leq k \leq N_{\mu}-1}\log (\frac {k}{k+1})\leq \sum_{N_0 \leq k \leq N_{\mu}-1}\frac {1}{k+1}+O(1)=O(1).
\end{equation}

   We shall see in what follows that the relation \eqref{3.6} implies that there is no harm studying $\log N_0$ instead of $\log N_{\mu}$. So from now on we shall concentrate on finding the smallest $k$ such that $h_k(1/\mu) > C_0$.
\section{Heuristic Treatment}
\label{sec 4} \setcounter{equation}{0}
   Our problem is, roughly, to determine how many steps we have to take in our recurrence \eqref{2.2} in order to push $h_k$ beyond the value of $C_0$,  assuming that $\mu$ is fixed, $0< \mu < (1-L/p)^{-p}$ and $\mu$ close to $(1-L/p)^{-p}$. Now assume we are able to neglect all the higher terms of the right-hand side expression in \eqref{3.2}, then we have a recurrence which can be written as
\begin{equation*}
  \Delta h=\frac {\Lambda_{k}}{\Lambda_{k+1}}\Big (\frac {\lambda_k}{\lambda_{k+1}}\Big)^{\frac 1{p-1}}\Big(h_{k}(\nu)+\frac {\lambda_k}{\Lambda_k}\nu^{\frac 1{p-1}}\Big)\Big (1+\frac {\lambda_{k}}{(p-1)\Lambda_k}\Big(h_{k}(\nu)+\frac {\lambda_k}{\Lambda_k}\nu^{\frac 1{p-1}}\Big)^{p-1} \Big )-h_k(\nu). 
\end{equation*}
  In view of \eqref{1.3'} and \eqref{3.5'}, we may further simplify the above recurrence to be the following:
\begin{equation*}
  \Delta h=\frac {1/L}{k+1}\Big (\frac {h^p_{k}(\nu)}{p-1}-\frac {p-L}{p-1}h_k(\nu)+\nu^{\frac 1{p-1}} \Big ).
\end{equation*}
   Next we consider $k$ as a continuous variable, and we replace the above by the corresponding differential equation, that is, we replace $\Delta h$ by $dh/dk$. Then we get
\begin{equation*}
  \frac {d \log (k+1)}{dh}=L\Big (\frac {h^p_{k}(\nu)}{p-1}-\frac {p-L}{p-1}h_k(\nu)+\nu^{\frac 1{p-1}} \Big )^{-1}.
\end{equation*}
   This suggests that if $N_0$ is the number of steps necessary to increase $h$ from $0$ to about $C_0$, then $\log N_0$ is roughly equal to
\begin{equation}
\label{4.1}
  L\int^{C_0}_0\frac {dh}{\frac {h^p}{p-1}-\frac {p-L}{p-1}h+\nu^{\frac 1{p-1}} }.
\end{equation}
   The integrand has its maximum at $h^{p-1}= (1-L/p)$. In the neighborhood of this maximum it can be approximated by
\begin{equation*}
  \frac p{2}(1-L/p)^{\frac {p-2}{p-1}}\Big(h-(1-L/p)^{\frac 1{p-1}}\Big)^2+\nu^{\frac 1{p-1}}-(1-L/p)^{\frac p{p-1}}.
\end{equation*}
   Therefore the value of \eqref{4.1} can be compared with
\begin{eqnarray*}
  && L\int^{+\infty}_{-\infty}\frac {dh}{ \frac p{2}(1-L/p)^{\frac {p-2}{p-1}}\Big(h-(1-L/p)^{\frac 1{p-1}}\Big)^2+\nu^{\frac 1{p-1}}-(1-L/p)^{\frac p{p-1}}} \\
&=& L \pi\Big (\frac p{2}(1-L/p)^{\frac {p-2}{p-1}}\Big (\nu^{\frac 1{p-1}}-(1-L/p)^{\frac p{p-1}} \Big )\Big )^{-1/2} .
\end{eqnarray*}
   From this we see that for $\mu < (1-L/p)^{-p}, \mu \rightarrow (1-L/p)^{-p}$, we expect to have
\begin{equation}
\label{4.2}
  \log N_{\mu}=L\pi\Big (\frac p{2}(1-L/p)^{\frac {p-2}{p-1}}\Big ((1/\mu)^{\frac 1{p-1}}-(1-L/p)^{\frac p{p-1}} \Big )\Big )^{-1/2}+O(1).
\end{equation}
   From this we see that if $ \mu \rightarrow (1-L/p)^{-p}$, then $\log N_{\mu}$ tends to infinity. This also implies that for the sequence $\{ \mu_k \}$ defined as in \eqref{3.1'}, one must have $\lim_{k \rightarrow +\infty}\mu_k =(1-L/p)^{-p}$. For otherwise, the sequence $\{ \mu_k \}$ is bounded above by a constant $<(1-L/p)^{-p}$ and on taking any $\mu$ greater than this constant (and less than $(1-L/p)^{-p}$), then the left-hand side of \eqref{4.2} becomes infinity (by our definition of $N_{\mu}$) but the right-hand side of \eqref{4.2} stays bounded, a contradiction. 

   Note that if $\mu=\mu_N$, then $N_{\mu}=N$, it follows from \eqref{4.2} that
\begin{equation*}
  (\frac {1}{\mu_N})^{\frac 1{p-1}}=(1-L/p)^{\frac p{p-1}} + \frac {2L^2\pi^2}{p}(1-L/p)^{-\frac {p-2}{p-1}}\Big ( \log N+O(1) \Big )^{-2}.
\end{equation*}
   It is easy to see that the above leads to the following asymptotic expression for $\mu_N$:
\begin{equation*}
  \mu_N =(1-L/p)^{-p} - \frac {2L^2\pi^2(1-L/p)^{-2-p}}{q(\log N)^2}+O\Big (\frac 1{(\log N)^3} \Big ).
\end{equation*}

   There are various doubtful steps in our argument above, but the only one that presents a serious difficulty is the omitting of all the other terms of the right-hand side expression of \eqref{3.2}. Certainly those terms can be expected to give only a small contribution if $k$ is large but the question is whether this contribution is small compared to $h^p/(p-1)-(p-L)h/(p-1)+\nu^{\frac 1{p-1}}$. The latter expression can be small if both $h^{p-1}- (1-L/p)$ and $\mu-(1-L/p)^{-p}$ are small, and it is especially in that region that the integrand of \eqref{4.1} produces its maximal effect.

\section{Lemmas}
\label{sec 5} \setcounter{equation}{0}
\begin{lemma}
\label{lem5.1}
  For any given number $\eta>0, 0< \epsilon <(1-L/p)^{1/(p-1)}$, one can find an integer $k_0 >\eta$ and a number $\beta$, $(1-L/p)^{-p}-1<\beta<(1-L/p)^{-p}$ such that for $\beta < \mu \leq (1-L/p)^{-p}$, 
\begin{equation}
\label{5.00}
  \Big(\frac {p-L}{p} \Big )^{1/(p-1)}-\epsilon < h_{k_0}(1/\mu) <  \Big(\frac {p-L}{p} \Big )^{1/(p-1)}-\frac {1}{2} \frac {\lambda_{k_0}}{\Lambda_{k_0}} \Big(\frac {p-L}{p} \Big )^{1/(p-1)}. 
\end{equation} 
\end{lemma}
\begin{proof}
  Note first that by \eqref{2.4} (with $\Lambda_0=0$) that
\begin{equation*}
  0 \leq h_k((1-L/p)^{-p}) \leq \frac {\Lambda_{k-1}}{\Lambda_k} \Big(\frac {p-L}{p} \Big )^{1/(p-1)}.
\end{equation*}
  Let $k_1$ be an integer so that for all $k \geq k_1$,
\begin{equation*}
   \frac {\Lambda_{k-1}}{\Lambda_k}\Big(\frac {p-L}{p} \Big )^{1/(p-1)} > \Big(\frac {p-L}{p} \Big )^{1/(p-1)}-\epsilon.
\end{equation*}
   We may assume that $k \geq k_1$ from now on and note that not all $h_k((1-L/p)^{-p})$ are $ \leq (1-L/p)^{1/(p-1)}-\epsilon$. Otherwise, it follows from \eqref{3.2}, \eqref{4.0}, \eqref{1.3'} and \eqref{3.5'} that
\begin{eqnarray*}
 && h_{k+1}((1-L/p)^{-p})-h_k((1-L/p)^{-p}) \\
&\geq & \frac {1/L}{k+1}\Big (\frac {h^p_{k}((1-L/p)^{-p})}{p-1}-\frac {p-L}{p-1}h_k((1-L/p)^{-p})+(1-L/p)^{\frac {p}{p-1}} \Big )+O(\frac {1}{k^2}).
\end{eqnarray*}
  Note that if $h_k(1-L/p)^{-p}) \leq (1-L/p)^{1/(p-1)}-\epsilon$ then
\begin{eqnarray*}
 && \frac {h^p_{k}((1-L/p)^{-p})}{p-1}-\frac {p-L}{p-1}h_k((1-L/p)^{-p})+(1-L/p)^{\frac {p}{p-1}} \\
&\geq & \frac {\Big ((1-L/p)^{1/(p-1)}-\epsilon \Big )^p}{p-1}-\frac {p-L}{p-1}\Big ((1-L/p)^{1/(p-1)}-\epsilon \Big )+(1-L/p)^{\frac {p}{p-1}}>0.
\end{eqnarray*}
  As $\sum^{\infty}_{k=k_1}(k+1)^{-1}=+\infty$, this leads to a contradiction since $h_k((1-L/p)^{-p})$ is bounded above by \eqref{2.4} for any $k$. Thus there is an integer $k_0 > \eta$ for which
\begin{eqnarray*}
 &&  \Big(\frac {p-L}{p} \Big )^{1/(p-1)}-\epsilon< h_{k_0}((1-L/p)^{-p}) \leq \frac {\Lambda_{k_0-1}}{\Lambda_{k_0}}\Big(\frac {p-L}{p} \Big )^{1/(p-1)}  \\
 &&< \Big(\frac {p-L}{p} \Big )^{1/(p-1)}-\frac {1}{2} \frac {\lambda_{k_0}}{\Lambda_{k_0}} \Big(\frac {p-L}{p} \Big )^{1/(p-1)}. 
\end{eqnarray*}
   Having fixed $k_0$ this way, we remark that $h_{k_0}(1/\mu)$ is continuous at $\mu=(1-L/p)^{-p}$ and the lemma follows.
\end{proof}

\begin{lemma}
\label{lem5.2}
  There exist numbers $\beta$, $(1-L/p)^{-p}-1<\beta<(1-L/p)^{-p}$, and $c>0$, $0<\delta <1$ such that for all $\mu$ satisfying $\beta < \mu \leq (1-L/p)^{1/(p-1)}$, and for all $k$ satisfying $1 \leq k \leq N_{\mu}$ ($N_{\mu}$ is the breakdown index) we have  
\begin{equation}
\label{5.0}
 \frac {h^p_{k}(1/\mu)}{p-1}-\frac {p-L}{p-1}h_k(1/\mu)+(1/\mu)^{\frac 1{p-1}}  > c \Big ( \frac {\Lambda_k}{\lambda_k} \Big )^{-\delta}. 
\end{equation} 
\end{lemma}
\begin{proof}
  We apply Lemma \ref{lem5.1} with $\eta$ large enough so that the following inequality holds for any integer $k \geq \eta$ and $(1-L/p)^{-p}-1 < \mu \leq (1-L/p)^{-p}$:
\begin{equation}
\label{5.1}
(1-L/p)^{1/(p-1)}+\frac {\lambda_k}{\Lambda_k}(1/\mu)^{\frac 1{p-1}} < 1. 
\end{equation}  
  We shall also choose $\epsilon$ small enough so that we obtain values of $k_0$ and $\beta$. Without loss of generality, we may assume $\mu < (1-L/p)^{-p}$ and for the time being we keep $\mu$ fixed ($\beta < \mu < (1-L/p)^{-p}$) and we write $h_k$ instead of $h_k(\mu)$.

  As we remarked in Section \ref{sec 3}, the sequence $h_{k_0}, h_{k_0+1}, \ldots$ is increasing, possibly until breakdown. We shall now first consider those integers $k \geq k_0$ for which $h_{k}< (1-L/p)^{1/(p-1)}$. For those $k$, it follows from \eqref{3.2}, \eqref{1.3'} and \eqref{3.5'} that
\begin{eqnarray}
 &&  h_{k+1}-h_k = \frac {\lambda_k}{\Lambda_{k+1}}\Big ( \frac {h^p_k}{p-1}-\frac {p-L}{p-1}h_k+(1/\mu)^{\frac 1{p-1}} \Big )+O(\frac {\lambda^2_k}{\Lambda_{k}\Lambda_{k+1}}) \nonumber \\
\label{5.2}
&&<    \frac {\lambda_k}{\Lambda_{k+1}}\Big ( B \Big ((1-L/p)^{1/(p-1)}-h_{k}\Big)^2+(1/\mu)^{\frac 1{p-1}}- (1-L/p)^{p/(p-1)}+ C_3 \frac {\lambda_k}{\Lambda_k}  \Big ),
\end{eqnarray}
  for some constant $C_3>0$, where
\begin{equation*}
  B=\frac p{2}\Big ( 1-L/p \Big )^{(p-2)/(p-1)}.
\end{equation*}
  We have by Lemma \ref{lem5.1}, $0<(1-L/p)^{1/(p-1)} -h_k < \epsilon$, and therefore we can replace \eqref{5.2} by the linear recurrence relation
\begin{equation}
\label{5.3}
  h_{k+1}-h_k < \frac {\lambda_k}{\Lambda_{k+1}}\Big ( \epsilon B \Big((1-L/p)^{1/(p-1)}-h_{k}\Big)+(1/\mu)^{\frac 1{p-1}}- (1-L/p)^{p/(p-1)}+C_3 \frac {\lambda_k}{\Lambda_k} \Big ).
\end{equation}
  We now let $\epsilon_1= \epsilon B$ and put
\begin{equation}
\label{5.3'}
   \epsilon_1  \Big((1-L/p)^{1/(p-1)}-h_{k}\Big)+(1/\mu)^{\frac 1{p-1}}- (1-L/p)^{p/(p-1)} =t_k,
\end{equation}
   so that it follows from \eqref{5.3} that
\begin{equation*}
   t_{k+1} > t_k\Big (1-\frac {\epsilon_1 \lambda_k}{\Lambda_{k+1}} \Big )-\frac {\epsilon_1  C_3\lambda^2_k}{\Lambda_{k}\Lambda_{k+1}} .
\end{equation*}
  We may assume our $\epsilon$ is so chosen so that $0< \epsilon< 1/(2p)$ and that $0<\epsilon_1<1/3$ and note that we have $L \leq 1$ by \eqref{5} so that when $p \geq 2$, we have $(1-L/p)^{1/(p-1)}>\epsilon$ so that by Lemma \ref{lem5.1} that $h_{k_0} >0$. Now it follows from $1-\epsilon x > (1-x)^{\epsilon}, 0<x<1$ that
\begin{equation*}
   t_{k+1} > t_k(\Lambda_k)^{\epsilon_1}(\Lambda_{k+1})^{-\epsilon_1}-\frac {\epsilon_1  C_3\lambda^2_k}{\Lambda_{k}\Lambda_{k+1}} 
=t_{k}(\frac {\Lambda_{k}}{\lambda_k})^{\epsilon_1}(\frac {\Lambda_{k+1}}{\lambda_k})^{-\epsilon_1}-\frac {\epsilon_1  C_3\lambda^2_k}{\Lambda_{k}\Lambda_{k+1}}.
\end{equation*}
  It follows from \eqref{1.6} that the sequence $\{ \Lambda_k/\lambda_k \}^{\infty}_{k=1}$ is increasing and we deduce that
\begin{equation*}
   t_{k+1} > t_{k_0}(\frac {\Lambda_{k_0}}{\lambda_{k_0}})^{\epsilon_1}(\frac {\Lambda_{k+1}}{\lambda_k})^{-\epsilon_1}-\frac {\epsilon_1  C_3\lambda^2_k}{\Lambda_{k}\Lambda_{k+1}},
\end{equation*}
   Note that $h_{k_0}$ is bounded below by Lemma \ref{lem5.1} and it follows from \eqref{1.3'} and \eqref{3.5'} that we may take $\eta$ large enough so that
we have
\begin{equation}
\label{5.6}
   t_{k+1} > t_{k_0}(\frac {\Lambda_{k_0}}{\lambda_{k_0}})^{\epsilon_1}(\frac {\Lambda_{k+1}}{\lambda_k})^{-\epsilon_2}
\end{equation}
  for some positive constant $\epsilon_2$.  As $t_{k_0} >0$, the above implies $t_k>0$ for all $k$ under consideration. We want the above to hold for all $k$ under consideration, i.e. for all $k$ for which $h_k < (1-L/p)^{1/(p-1)}$. This is certainly satisfied if $t_k>(1/\mu)^{\frac 1{p-1}}- (1-L/p)^{p/(p-1)}$, and \eqref{5.6} guarantees that this is true as long as the right-hand side expression of \eqref{5.6} is $>(1/\mu)^{\frac 1{p-1}}- (1-L/p)^{p/(p-1)}$. Therefore
\begin{equation}
\label{5.7}
  t_k \geq t_{k_0}(\frac {\Lambda_{k_0}}{\lambda_{k_0}})^{\epsilon_1}(\frac {\Lambda_{k}}{\lambda_{k-1}})^{-\epsilon_2}
\end{equation}
   for all $k > k_0$ satisfying
\begin{equation}
\label{5.8}
  \frac {\Lambda_{k}}{\lambda_{k-1}} < \Big (\frac {\Lambda_{k_0}}{\lambda_{k_0}} \Big )^{\epsilon_1/\epsilon_2}t^{1/\epsilon_2}_{k_0}\Big ((1/\mu)^{\frac 1{p-1}}- (1-L/p)^{p/(p-1)} \Big )^{-1/\epsilon_2}.
\end{equation}
   As $h_k < (1-L/p)^{1/(p-1)}$, we are sure that no breakdown occurs in this range by \eqref{3.1} and \eqref{5.1}.

   Now we return to the discussion on \eqref{5.0}. When $0< h <(1-L/p)^{1/(p-1)}$, we note that it is easy to see that there exists a constant $c_1=B/(p-1)>0$ such that for $0< x <(1-L/p)^{1/(p-1)}$,
\begin{equation*}
  \frac {x^p}{p-1}-\frac {p-L}{p-1}x \geq c_1\Big ((1-L/p)^{1/(p-1)}-x \Big )^2-(1-L/p)^{p/(p-1)}.
\end{equation*}
 Note also that $0 <(1/\mu)^{\frac 1{p-1}}- (1-L/p)^{p/(p-1)}<1$,and that
\begin{eqnarray*}
 && \frac {h^p}{p-1}-\frac {p-L}{p-1}h+(1/\mu)^{\frac 1{p-1}} \\
&>&(1/\mu)^{\frac 1{p-1}}- (1-L/p)^{p/(p-1)}+ c_1\Big ((1-L/p)^{1/(p-1)}-h \Big )^2 \\
&>&  \Big ( (1/\mu)^{\frac 1{p-1}}- (1-L/p)^{p/(p-1)} \Big )^2+ c_1\Big ((1-L/p)^{1/(p-1)}-h \Big )^2 \\
& >& \frac 1{2}\Big ((1/\mu)^{\frac 1{p-1}}- (1-L/p)^{p/(p-1)}+\sqrt{c_1}\Big ( (1-L/p)^{1/(p-1)}-h \Big ) \Big )^2,
\end{eqnarray*} 
  where the last inequality above follows from $u^2+c_1v^2 = u^2+(\sqrt{c_1}v)^2 \geq (u+\sqrt{c_1}v)^2/2$ for $u, v>0$. Apply this with $h=h_k$ and note that $\sqrt{c_1} \geq \epsilon B$ since $0< \epsilon < 1/(2p)$, so that it follows from \eqref{5.3'} and \eqref{5.7} that
\begin{equation*}
   \sqrt{c_1}\Big ( (1-L/p)^{1/(p-1)}-h_k \Big )+(1/\mu)^{\frac 1{p-1}}- (1-L/p)^{p/(p-1)} > t_k \geq t_{k_0}(\frac {\Lambda_{k_0}}{\lambda_{k_0}})^{\epsilon_1}(\frac {\Lambda_{k}}{\lambda_{k-1}})^{-\epsilon_2},
\end{equation*}
   This implies that the left-hand side of \eqref{5.0} is at least
\begin{equation*}
    \frac {t^2_{k_0}}{2}(\frac {\Lambda_{k_0}}{\lambda_{k_0}})^{2\epsilon_1}(\frac {\Lambda_{k}}{\lambda_{k-1}})^{-2\epsilon_2}.
\end{equation*}
   This holds for $k$ when \eqref{5.8} is satisfied. It follows from \eqref{1.3} that 
$\lambda_{k}/\lambda_{k-1}$ is bounded above for any $k \geq 2$. Let $c_2$ denote such an upper bound and we conclude that 
the left-hand side of \eqref{5.0} is at least
\begin{equation*}
    \frac {t^2_{k_0}}{2c^{2\epsilon_2}_2}(\frac {\Lambda_{k_0}}{\lambda_{k_0}})^{2\epsilon_1}(\frac {\Lambda_{k}}{\lambda_{k}})^{-2\epsilon_2}:=c_3(\frac {\Lambda_{k}}{\lambda_{k}})^{-2\epsilon_2}.
\end{equation*}
   Other $k$'s do not cause much trouble. First, for the values $1 \leq k < k_0$, we have 
\begin{equation*}
 h_k(\mu) \leq h_{k_0}(\mu) <  \Big(\frac {p-L}{p} \Big )^{1/(p-1)}-\frac {1}{2} \frac {\lambda_{k_0}}{\Lambda_{k_0}} \Big(\frac {p-L}{p} \Big )^{1/(p-1)}
\end{equation*}
   by Lemma \ref{lem5.1} and the fact that $h_k$ increases as $k$ increases. It follows that
\begin{eqnarray*}
&& \frac {h^p_k}{p-1}-\frac {p-L}{p-1}h_k+(1/\mu)^{\frac 1{p-1}} >   c_1\Big((1-L/p)^{1/(p-1)} - h_k \Big )^2 >  \frac {c_1\lambda^2_{k_0}}{4\Lambda^2_{k_0}}\Big(\frac {p-L}{p} \Big )^{2/(p-1)} \\
&&\geq \frac {c_1\lambda^2_{k_0}}{4\Lambda^2_{k_0}}\Big(\frac {p-L}{p} \Big )^{2/(p-1)}(\frac {\Lambda_{k}}{\lambda_{k}})^{-2\epsilon_2}:=c_4(\frac {\Lambda_{k}}{\lambda_{k}})^{-2\epsilon_2}.
\end{eqnarray*} 
   Now, for the remaining case $k_0 \leq k \leq  N_{\mu}$ (which is empty if $\mu=(1-L/p)^{-p}$) such that
\begin{equation*}
   \frac {\Lambda_{k}}{\lambda_{k-1}} \geq \Big (\frac {\Lambda_{k_0}}{\lambda_{k_0}} \Big )^{\epsilon_1/\epsilon_2}t^{1/\epsilon_2}_{k_0}\Big ((1/\mu)^{\frac 1{p-1}}- (1-L/p)^{p/(p-1)} \Big )^{-1/\epsilon_2},
\end{equation*}
 we note that as the sequence $\{ \lambda_k \}^{\infty}_{k=1}$ is non-decreasing and the sequence $\{ \Lambda_k/\lambda_k \}^{\infty}_{k=1}$ is increasing by \eqref{1.6}, if the above inequality holds for some $k'$ then it holds for all $k \geq k'$ and in this case we use that 
\begin{equation*}
 \frac {h^p}{p-1}-\frac {p-L}{p-1}h+(1/\mu)^{\frac 1{p-1}} > (1/\mu)^{\frac 1{p-1}}- (1-L/p)^{p/(p-1)}
\end{equation*} 
  for all $h$ to see that the left-hand side of \eqref{5.0} is at least
\begin{equation*}
  \Big (\frac {\Lambda_{k_0}}{\lambda_{k_0}} \Big )^{\epsilon_1}\frac {t_{k_0}}{c^{\epsilon_2}_2}(\frac {\Lambda_{k}}{\lambda_{k}})^{-2\epsilon_2}:=c_5(\frac {\Lambda_{k}}{\lambda_{k}})^{-2\epsilon_2}.
\end{equation*}
  In all three cases the constants are independent of $\mu$ and $k$, so on letting $c=\min (c_3, c_4, c_5)$ and $\delta=2\epsilon_2$ completes the proof of the lemma.
\end{proof}

\begin{lemma}
\label{lem5.3}
  There exist numbers $\beta$, $(1-L/p)^{-p}-1<\beta<(1-L/p)^{-p}$ such that for all $\mu$ satisfying $\beta < \mu < (1-L/p)^{-p}$ there exists an index $N <  N_{\mu}$ with $h_N > C_0$ with $C_0$ defined as in Section \ref{sec 3}.
\end{lemma}
\begin{proof}
  We apply Lemma \ref{lem5.1} with $\eta$ large enough and some $0< \epsilon <1$, so that the following estimation holds for any integer $k \geq \eta$ and $\mu > (1-L/p)^{-p}-1$:
\begin{equation}
\label{5.12}
\Big(\frac {\Lambda_k}{\lambda_{k}}\Big )^{\frac 1{p-1}}-\frac {\lambda_{k}}{\Lambda_{k}} \Big ( \frac 1{ \mu} \Big )^{\frac 1{p-1}} > C_0+1.
\end{equation}  
  Moreover, we can also take $\eta$ large enough so that for any integer $k \geq \eta$ and $\mu > (1-L/p)^{-p}-1$, if $h_k(1/\mu)<C_0+1$, then
\begin{equation}
\label{5.14}
h_{k+1}-h_k = \frac {\lambda_k}{\Lambda_{k+1}}\Big ( \frac {h^p_k}{p-1}-\frac {p-L}{p-1}h_k+(1/\mu)^{\frac 1{p-1}} \Big )+O(\frac {\lambda^2_k}{\Lambda_{k}\Lambda_{k+1}})<1.
\end{equation}  
  Note that it follows from \eqref{3.2}, \eqref{1.3'} and \eqref{3.5'} and our assumption on $h_k$ that the above requirement can be satisfied. Lemma \ref{lem5.1} now provides us with $k_0>\eta$ and $\beta$ such that \eqref{5.00} holds. We now consider the numbers $h_{k_0}, h_{k_0+1}, \ldots$ as far as they are $< C_0+1$. If $k \geq k_0$, $h_k< C_0+1$, we have $k < N_{\mu}$ by \eqref{5.12} and by our definition of the breakdown index (see \eqref{3.1}). It also follows from \eqref{3.2}, \eqref{1.3'} and \eqref{3.5'} that
\begin{eqnarray*}
  h_{k+1}-h_k &=& \frac {\lambda_k}{\Lambda_{k+1}}\Big ( \frac {h^p_k}{p-1}-\frac {p-L}{p-1}h_k+(1/\mu)^{\frac 1{p-1}} \Big )+O(\frac {\lambda^2_k}{\Lambda_{k}\Lambda_{k+1}}) \\
& \geq &  \frac {\lambda_k}{\Lambda_{k+1}}\Big ( (1/\mu)^{\frac 1{p-1}}- (1-L/p)^{p/(p-1)}\Big )+O(\frac {\lambda^2_k}{\Lambda_{k}\Lambda_{k+1}}).
\end{eqnarray*}
  The lower bound above shows that not for all $k \geq k_0$ we have $h_k \leq C_0$, since $\sum^{+\infty}_{k_0}(h_{k+1}-h_k)$ would diverge in view of \eqref{3.5'}.

   It follows from \eqref{5.14} that if we let $h_{k_1}$ be the last one below $C_0$, then $h_{k_1+1}$ is still below $C_0+1$ so that we can take $N=k_1+1$ here and this completes the proof.
\end{proof}

\section{Proof of Theorem \ref{thm2}}
\label{sec 6} \setcounter{equation}{0}
   As suggested by the discussion in Section \ref{sec 4}, we shall study $\theta(h_k)$, where $\theta$ is defined by
\begin{equation*}
   \theta(y)=\int^{y}_{0}\frac {dx}{\frac {x^p}{p-1}-\frac {p-L}{p-1}x+(1/\mu)^{\frac 1{p-1}}}.
\end{equation*}
   We first simplify the recurrence formula \eqref{3.2}. Assuming
\begin{equation}
\label{6.1}
    (1-L/p)^{-p}-1 < \mu \leq (1-L/p)^{-p}, \hspace{0.1in} h_k \leq C_0,
\end{equation}
  where $C_0$ is defined as in Section \ref{sec 3}. We may also assume $k$ is large enough so that \eqref{3.1} is not satisfied. We have, by \eqref{3.2} and Taylor expansions,
\begin{equation*}
   h_{k+1}-h_k=\frac {\lambda_k}{\Lambda_{k+1}}\Big ( \frac {h^p_k}{p-1}-\frac {p-L}{p-1}h_k+(1/\mu)^{\frac 1{p-1}}+\gamma_k  \Big ),
\end{equation*}
   where
\begin{equation*}
  |\gamma_k|  \leq C_2\frac {\lambda_k}{\Lambda_{k}},
\end{equation*}
   for some constant $C_2>0$. It follows from this that there exists a constant $C_3>0$ such that 
\begin{equation*}
    \Big| h_{k+1}-h_k \Big| \leq C_3\frac {\lambda_k}{\Lambda_{k+1}}.
\end{equation*}
   We then deduce easily from above that for $h_k \leq x \leq h_{k+1}$,
\begin{equation*}
   \Big |\frac {x^p}{p-1}-\frac {p-L}{p-1}x-\Big(\frac {h^p_k}{p-1}-\frac {p-L}{p-1}h_k \Big ) \Big | \leq C_4\frac {\lambda_k}{\Lambda_{k+1}} \leq C_4\frac {\lambda_k}{\Lambda_{k}},
\end{equation*}
   where $C_4>0$ is a constant not depending on $\mu$ or $k$ (still assuming \eqref{6.1}).

   We now apply the mean value theorem to get:
\begin{equation*}
   \theta(h_{k+1})-\theta(h_{k})=(h_{k+1}-h_k)\theta '(x)
\end{equation*}
   with some $x$ in between $h_k$ and $h_{k+1}$. Hence it follows from our discussion above that
\begin{equation*}
   \theta(h_{k+1})-\theta(h_{k})=\frac {\lambda_k}{\Lambda_{k+1}}\frac {H+\gamma_k}{H+\gamma '_k},
\end{equation*}
  where 
\begin{equation*}
  H=\frac {h^p_k}{p-1}-\frac {p-L}{p-1}h_k+(1/\mu)^{\frac 1{p-1}}, \hspace{0.1in} |\gamma_k| \leq C_2\frac {\lambda_k}{\Lambda_{k}}, \hspace{0.1in} |\gamma '_k| \leq C_4\frac {\lambda_k}{\Lambda_{k}}.
\end{equation*}
   
   We now apply Lemma \ref{lem5.2} to conclude that there exists a $\beta_1$ with $(1-L/p)^{-p}-1 < \beta_1 < (1-L/p)^{-p}$ and a $c>0$, $0<\delta <1$  such that for all $\mu$ satisfying $\beta < \mu \leq (1-L/p)^{-p}$, and for all $k$ satisfying $1 \leq k \leq N_{\mu}$, we have
\begin{equation*}
  H > c\Big ( \frac {\Lambda_k}{\lambda_k} \Big )^{-\delta}.
\end{equation*}   
   This implies that
\begin{equation*}
 |\gamma_k| \leq \frac {C_2}{c} \Big (\frac {\lambda_k}{\Lambda_{k}} \Big )^{1-\delta}H, \hspace{0.1in} |\gamma '_k| \leq \frac {C_4}{c} \Big (\frac {\lambda_k}{\Lambda_{k}} \Big )^{1-\delta}H.
\end{equation*} 
   Note it follows from \eqref{3.5'} that
\begin{equation*}
  \frac {\lambda_k}{\Lambda_{k+1}}-\frac {\lambda_k}{\Lambda_{k}}=-\frac {\lambda_k\lambda_{k+1}}{\Lambda_{k}\Lambda_{k+1}}=O(\frac 1{k^2}).
\end{equation*}
   It follows from this and \eqref{3.5'} that we can find an integer $m$, independent of $\mu$ such that for $k>m, h_k < C_0$, we have
\begin{equation*}
   \theta(h_{k+1})-\theta(h_{k})=\frac {\lambda_k}{\Lambda_{k}}\frac {H+\gamma_k}{H+\gamma '_k}+O(\frac 1{k^2})=\frac {1}{Lk}+O \Big (\frac 1{k^2}+\frac 1{k^{2-\delta}} \Big ).
\end{equation*}
   We recast the above as
\begin{equation*}
   \Big|\theta(h_{k+1})-\theta(h_{k})- \log (1+1/k)/L \Big | =O \Big (\frac 1{k^2}+\frac 1{k^{2-\delta}} \Big ).
\end{equation*}
   Now assuming $\mu < (1-L/p)^{-p}$, we take the sum over the values $m \leq k <N_0$, where $N_0$ is the first index with $h_{N_0} > C_0$ (see Lemma \ref{lem5.3}). This gives us
\begin{equation*}
    \Big|\theta(h_{N_0})- \log N_0/L  \Big| =O (1)+\log m/L+\theta(h_m).
\end{equation*}

   By Lemma \ref{lem5.1}, for any $\eta>M$, there exists $\beta_2, \beta_1 < \beta_2 <(1-L/p)^{-p}$ and $k_0>\eta$ so that 
\begin{equation*}
 h_{k_0}(\mu) < \Big(\frac {p-L}{p} \Big )^{1/(p-1)}-\frac {1}{2} \frac {\lambda_{k_0}}{\Lambda_{k_0}} \Big(\frac {p-L}{p} \Big )^{1/(p-1)}.
\end{equation*}
 We now further take the integer $m$ to be equal to this $k_0$. Thus, the maximum of the integrand in $\theta(h_m)$ is attained at $x=h_m$ and that by Taylor expansion again,
\begin{eqnarray*}
 &&  \frac {h^p_m}{p-1}-\frac {p-L}{p-1}h_m+(1/\mu)^{\frac 1{p-1}} \\
 & \geq & \frac {1}{p-1}\Big ( \Big(\frac {p-L}{p} \Big )^{1/(p-1)}-\frac {1}{2} \frac {\lambda_{m}}{\Lambda_{m}} \Big(\frac {p-L}{p} \Big )^{1/(p-1)} \Big )^p \\
&& -\frac {p-L}{p-1}\Big ( \Big(\frac {p-L}{p} \Big )^{1/(p-1)}-\frac {1}{2} \frac {\lambda_{m}}{\Lambda_{m}} \Big(\frac {p-L}{p} \Big )^{1/(p-1)} \Big )+(1-L/p)^{\frac p{p-1}}  \\
 & \geq &  \frac {p}{8}\Big(\frac {p-L}{p} \Big )^{p/(p-1)}\frac {\lambda^2_{m}}{\Lambda^2_{m}}\Big( 1-\frac {\lambda_{m}}{2\Lambda_{m}} \Big)^{p-2}.
\end{eqnarray*}
  It follows that
\begin{eqnarray*}
 && \theta(h_m)=\int^{h_m}_{0}\frac {dx}{\frac {x^p}{p-1}-\frac {p-L}{p-1}x+(1/\mu)^{\frac 1{p-1}}} \\
&& \leq \Big(\frac {p-L}{p} \Big )^{1/(p-1)}\Big ( \frac {p}{8}\Big(\frac {p-L}{p} \Big )^{p/(p-1)}\frac {\lambda^2_{m}}{\Lambda^2_{m}}\Big( 1-\frac {\lambda_{m}}{2\Lambda_{m}} \Big)^{p-2} \Big )^{-1}=O(1).
\end{eqnarray*}
   We deduce from this that
\begin{equation}
\label{6.2}
    \Big|\theta(h_{N_0})- \log N_0/L  \Big| =O (1).
\end{equation}

   It is not difficult to find the asymptotic behavior of $\theta(\infty)$. If $\mu <(1-L/p)^{-p}$, $\mu \rightarrow (1-L/p)^{-p}$, then routine methods (cf. Sec. \ref{sec 4}) lead to
\begin{equation*}
   \theta(\infty)=\int^{\infty}_{0}\frac {dx}{\frac {x^p}{p-1}-\frac {p-L}{p-1}x+(1/\mu)^{\frac 1{p-1}}}=\pi\Big (\frac p{2}(1-L/p)^{\frac {p-2}{p-1}}\Big ((1/\mu)^{\frac 1{p-1}}-(1-L/p)^{\frac p{p-1}} \Big )\Big )^{-1/2} +O(1).
\end{equation*}
   It is also easy to see that $\theta(\infty)-\theta(C_0) = O(1)$. As $h_{N_0} \geq C_0$, we have $\theta(C_0) \leq \theta(h_{N_0}) < \theta(\infty)$.  
It follows from \eqref{6.2} that
\begin{equation*}
   \log N_0=L\pi\Big (\frac p{2}(1-L/p)^{\frac {p-2}{p-1}}\Big ((1/\mu)^{\frac 1{p-1}}-(1-L/p)^{\frac p{p-1}} \Big )\Big )^{-1/2} +O(1).
\end{equation*}
  According to \eqref{3.6} and our discussion in Section \ref{sec 4}, this completes the proof of \eqref{4.2} and it was already shown there that \eqref{4.2} leads to our assertion for Theorem \ref{thm2}.
  
\section{An Application of Theorem \ref{thm2}}
\label{sec 7} \setcounter{equation}{0}
   As we mentioned as the beginning of the paper that the main cases we are interested in are the cases $\lambda_k=k^{\alpha}$ and here we may assume  $\alpha \geq 1$. Of all the conditions in the statement of Theorem \ref{thm2}, it is easy to check \eqref{5}-\eqref{1.6} are satisfied with $L=1/(\alpha+1)$ for our cases here (for example, see \cite[Section 7]{G6}). It is easy to see that \eqref{1.4} follows from $g_k(1/p) \geq 0$ where (with $\lambda_0=0$)
\begin{equation*}
  g_k(x)=\frac {\Lambda_{k}}{\Lambda_{k+1}}\Big ( \frac {\Lambda_{k+1}/\lambda_{k+1}}{\Lambda_{k}/ \lambda_{k}} \Big )^x-\frac {\lambda_{k}}{\Lambda_{k}}Lx-\frac {\Lambda_{k-1}}{\Lambda_{k}}.
\end{equation*}
   It follows from \eqref{1.6} and \eqref{5} that for $0 < x \leq 1$,
\begin{eqnarray*}
   g'_k(x) &=& \frac {\Lambda_{k}}{\Lambda_{k+1}}\log \Big ( \frac {\Lambda_{k+1}/\lambda_{k+1}}{\Lambda_{k}/ \lambda_{k}} \Big ) \Big ( \frac {\Lambda_{k+1}/\lambda_{k+1}}{\Lambda_{k}/ \lambda_{k}} \Big )^x-\frac {\lambda_{k}}{\Lambda_{k}}L \\
 & \leq & \frac {\lambda_{k}}{\lambda_{k+1}}\log \Big ( \frac {\Lambda_{k+1}/\lambda_{k+1}-\Lambda_{k}/ \lambda_{k}}{\Lambda_{k}/ \lambda_{k}}+1 \Big )-\frac {\lambda_{k}}{\Lambda_{k}}L \\
& \leq & \frac {\lambda_{k}}{\lambda_{k+1}}\Big ( \frac {L}{\Lambda_{k}/ \lambda_{k}}\Big )-\frac {\lambda_{k}}{\Lambda_{k}}L \leq 0,
\end{eqnarray*}
  where the last inequality above follows as $\{ \lambda_k \}^{\infty}_{k=1}$ is a non-decreasing positive sequence.
 
   We then deduce that for $0<x \leq 1$,
\begin{equation*}
   g_k(x) \geq g_k(1)=\frac {\lambda_{k}}{\lambda_{k+1}}-\frac {\lambda_{k}}{\Lambda_{k}}L-\frac {\Lambda_{k-1}}{\Lambda_{k}}.
\end{equation*}
   Note that $g_k(1)=0$ when $\lambda_k=k$, so that \eqref{1.4} is satisfied and we deduce the following
\begin{cor}
\label{cor1}
  Let $\lambda_k=k$ for $k \geq 1$. Then for $p \geq 2$, inequality \eqref{1} holds with the best constant satisfying:
\begin{equation*}
   \mu_N=(1-1/(2p))^{-p} - \frac {\pi^2(1-1/(2p))^{-2-p}}{2q(\log N)^2}+O\Big (\frac 1{(\log N)^3} \Big ).
\end{equation*}  
\end{cor}

     Things are more complicated for other $\alpha$'s in general. For example, one can see that $g(1)<0$ when $\alpha=3$ but $g_k(1/2)>0$ for $k \geq 2$ and as we have assumed $p \geq 2$, this implies Theorem \ref{thm2} for $\lambda_k=k^3$.  In fact,  a close look at the proof of Theorem \ref{thm2} shows that one only needs \eqref{1.4} to hold asymptotically, namely, for all large $k$'s. We shall leave the more general discussions to the reader.  

\end{document}